\documentclass[12pt,reqno]{amsart}
\usepackage{fullpage}
\usepackage{amsmath}
\usepackage{mathrsfs,amssymb,graphicx,verbatim,hyperref,mathtools}
\usepackage{paralist}
\renewcommand{\setminus}{\smallsetminus}
\usepackage{ifthen}

\newcommand{\ifsodaelse}[2]{\ifthenelse{\isundefined{\SODAF}}{#2}{#1}}
\ifsodaelse    {\usepackage{ltexpprt}\usepackage{amsmath}}{}
\usepackage{mathrsfs,amssymb,graphicx,verbatim}
\usepackage{paralist}

\newcommand\remove[1]{}

\newcommand{\rnote}[1]{}
\newcommand{\jnote}[1]{}

\newcommand{\sign}{\mathrm{sign}}
\renewcommand{\d}{\delta}
\newcommand{\D}{\mathbb D}

\newcommand{\e}{\varepsilon}
\newcommand{\R}{\mathbb{R}}

\newcommand{\E}{\mathbb{E}}

\newcommand{\N}{\mathbb{N}}

\newcommand{\C}{\mathbb{C}}

\newtheorem{theorem}{Theorem}[section]
\newtheorem{lemma}[theorem]{Lemma}

\newcommand{\eqdef}{\stackrel{\mathrm{def}}{=}}
\date{}

\renewcommand{\le}{\leqslant}
\renewcommand{\ge}{\geqslant}

\include{psfig}

\renewcommand{\epsilon}{\varepsilon}

\theoremstyle{remark}

\renewcommand{\phi}{\varphi}
\renewcommand{\S}{\mathbb{S}}
\renewcommand{\H}{\mathcal{H}}

\title{Krivine schemes are optimal}
\thanks{A.N. was supported by NSF grant CCF-0832795, BSF grant
 2010021, the Packard Foundation and the Simons Foundation. O. R. was supported by a European Research Council (ERC) Starting Grant.}

\author{Assaf Naor}

\author{Oded Regev}

\date{}

\begin{document}

\begin{abstract}
It is shown that for every $k\in \N$ there exists a Borel probability measure $\mu$ on $\{-1,1\}^{\R^{k}}\times \{-1,1\}^{\R^{k}}$ such that for every $m,n\in \N$ and $x_1,\ldots, x_m,y_1,\ldots,y_n\in \S^{m+n-1}$ there exist $x_1',\ldots,x_m',y_1',\ldots,y_n'\in \S^{m+n-1}$ such that if $G:\R^{m+n}\to \R^k$ is a random $k\times (m+n)$ matrix whose entries are i.i.d.\ standard Gaussian random variables then for all $(i,j)\in \{1,\ldots,m\}\times \{1,\ldots,n\}$ we have
\begin{equation*}
\E_G\left[\int_{\{-1,1\}^{\R^{k}}\times \{-1,1\}^{\R^{k}}}f(Gx_i')g(Gy_j')d\mu(f,g)\right]=\frac{\langle x_i,y_j\rangle}{(1+C/k)K_G},
\end{equation*}
where $K_G$ is the real Grothendieck constant and $C\in (0,\infty)$ is a universal constant. This establishes that Krivine's rounding method yields an arbitrarily good approximation of $K_G$.
\end{abstract}

\maketitle

\section{introduction}\label{sec:intro}

Grothendieck's inequality~\cite{Gro53,LP68} asserts that there exists a constant $K\in (0,\infty)$ such that for every $n,m\in \N$, every $m\times n$ matrix  $A=(a_{ij})$ with real entries, and every choice of unit vectors $x_1,\ldots,x_m,y_1,\ldots,y_n\in \S^{m+n-1}$, there are signs $\e_1,\ldots,\e_m,\d_1,\ldots,\d_n\in \{-1,1\}$ satisfying
\begin{equation}\label{eq:def gro ineq}
\sum_{i=1}^m\sum_{j=1}^n a_{ij}\langle x_i,y_j\rangle \le K \sum_{i=1}^m\sum_{j=1}^n a_{ij}\e_i\d_j.
\end{equation}
The infimum over those $K\in (0,\infty)$ for which~\eqref{eq:def gro ineq} holds true is called the (real) Grothendieck constant and is denoted $K_G$. The exact value of $K_G$ remains a long-standing mystery ever since Grothendieck's original 1953 work~\cite{Gro53}. The best known bounds on $K_G$ are due to~\cite{Ree91,BMMN11}.

We call a Borel probability measure $\mu$ on $\{-1,1\}^{\R^{k}}\times \{-1,1\}^{\R^{k}}$
 a $k$-dimensional {\em Krivine scheme} of quality $K\in (0,\infty)$ if for every $m,n\in \N$ and every choice of unit vectors $x_1,\ldots, x_m,y_1,\ldots,y_n\in \S^{m+n-1}$ there exist new unit vectors $x_1',\ldots,x_m',y_1',\ldots,y_n'\in \S^{m+n-1}$ such that if $G:\R^{m+n}\to \R^k$ is a random $k\times (m+n)$ matrix whose entries are i.i.d.\ standard Gaussian random variables then for all $(i,j)\in \{1,\ldots,m\}\times \{1,\ldots,n\}$ we have
\begin{equation}\label{eq:scheme identity}
\E_G\left[\int_{\{-1,1\}^{\R^{k}}\times \{-1,1\}^{\R^{k}}}f(Gx_i')g(Gy_j')d\mu(f,g)\right]=\frac{1}{K}\langle x_i,y_j\rangle.
\end{equation}
For~\eqref{eq:scheme identity} to make sense we need to require that the function of $G$ appearing under the expectation in~\eqref{eq:scheme identity}  is Lebesgue integrable. The required integrability will be immediate for the Krivine schemes that we discuss below, but in general this assumption should be taken to be part of the definition of a Krivine scheme.

The existence of  a $k$-dimensional  Krivine scheme of quality $K$ implies that $K_G\le K$. Indeed, since $f(Gx_i'),g(Gy_j')\in \{-1,1\}$ point-wise,
\begin{multline}\label{eq:apply scheme}
\sum_{i=1}^m\sum_{j=1}^n a_{ij}\langle x_i,y_j\rangle =K \E_G\left[\int_{\{-1,1\}^{\R^{k}}\times \{-1,1\}^{\R^{k}}}\left(\sum_{i=1}^m\sum_{j=1}^n a_{ij}f(Gx_i')g(Gy_j')\right)d\mu(f,g)\right]\\\le
K\max_{\e_1,\ldots,\e_m,\d_1,\ldots,\d_n\in \{-1,1\}}\sum_{i=1}^m\sum_{j=1}^n a_{ij}\e_i\d_j.
\end{multline}

A different way to describe the above reasoning is via the following procedure that, given unit vectors $x_1,\ldots,x_m,y_1,\ldots,y_n\in \S^{m+n-1}$, produces signs $\e_1,\ldots,\e_m,\d_1,\ldots,\d_n\in \{-1,1\}$ so as to satisfy the desired inequality~\eqref{eq:def gro ineq}. Identify $f\in \{-1,1\}^{\R^k}$ with the bipartition of $\R^k$ given by $f^{-1}(-1), f^{-1}(1)\subseteq \R^k$. Thus $\mu$ could be thought of as a carefully chosen notion of a random pair $(\{A,\R^k\setminus A\}, \{B,\R^k\setminus B\})$ of bipartitions of $\R^k$. The new unit vectors $x_1',\ldots,x_m',y_1',\ldots,y_n'\in \S^{m+n-1}$ should be thought of as the result of {\em preprocessing} the initial vectors  $x_1,\ldots, x_m,y_1,\ldots,y_n\in \S^{m+n-1}$, in anticipation of the ``random projection" step to come. In this step one chooses a random $k\times (m+n)$ Gaussian matrix $G$, and, independently of $G$, a pair of bipartitions $(\{A,\R^k\setminus A\}, \{B,\R^k\setminus B\})$ whose law is $\mu$. One then defines {\em random signs} $\e_1,\ldots,\e_m,\d_1,\ldots,\d_n\in \{-1,1\}$ by
$$
\e_i\eqdef \left\{\begin{array}{ll}1&\mathrm{if\ }Gx_i'\in A,\\
-1&\mathrm{if\ }Gx_i'\in \R^k\setminus A,\end{array}\right.\quad\mathrm{and}\quad \d_j\eqdef \left\{\begin{array}{ll}1&\mathrm{if\ }Gy_j'\in B,\\
-1&\mathrm{if\ }Gy_j'\in \R^k\setminus B.\end{array}\right.
$$
Inequality~\eqref{eq:apply scheme} means that for this choice of random signs the desired Grothendieck inequality~\eqref{eq:def gro ineq} holds {\em in expectation}, with respect to the randomness of the matrix $G$ and the pair of bipartitions  $(\{A,\R^k\setminus A\}, \{B,\R^k\setminus B\})$.

The preprocessing step of a Krivine scheme has been somewhat
suppressed in the above discussion because  it is essentially
determined by the measure $\mu$ itself. Indeed, by rotation
invariance the left hand side of~\eqref{eq:scheme identity} depends
only on $\langle x_i',y_j'\rangle$; say, it equals $\Psi(\langle
x_i',y_j'\rangle)$ for some function $\Psi:[-1,1]\to [-1,1]$. If $\Psi$
were invertible on $[-1/K,1/K]$ then~\eqref{eq:scheme identity} would become
$\langle x_i',y_j'\rangle=\Psi^{-1}(\langle x_i,y_j\rangle/K)$. This
means that the mutual angles between the new unit vectors (if they
indeed exist) are determined by~\eqref{eq:scheme identity} and the
initial configuration of vectors. Only these angles matter for the
purpose of applying a Krivine scheme to bound $K_G$.

We note that in the known examples of Krivine schemes (including those described below) the preprocessing step that associates new unit vectors $x_1',\ldots,x_m',y_1',\ldots,y_n'\in \S^{m+n-1}$ to the initial unit vectors $x_1,\ldots, x_m,y_1,\ldots,y_n\in \S^{m+n-1}$ is {\em oblivious} to the initial configuration in the following sense. There exist (nonlinear) mappings $S,T:\S^{\infty}\to \S^\infty$, where $\S^\infty$ is the unit sphere of $\ell_2$, such that if we identify the span of $T(x_1),\ldots,T(x_m),S(y_1),\ldots,S(y_n)$ with a subspace of $\R^{m+n}$ then~\eqref{eq:scheme identity} holds true with $x_i'=T(x_i)$ and $y_j'=S(y_j)$. We call such Krivine schemes {\em oblivious Krivine schemes}. It is possible to show that any Krivine scheme satisfying an additional continuity assumption is oblivious. For all purposes that we can imagine, one can take the definition of a Krivine scheme to include the assumption that it is oblivious. 


\begin{theorem}[Optimality of Krivine schemes]\label{thm:main}
For every $k\in \N$ there exists a $k$-dimensional oblivious Krivine scheme of quality $(1+O(1/k))K_G$. 
\end{theorem}

 One might initially believe that restricting attention to Krivine schemes in order to prove Grothendieck's inequality is too restrictive: the role of random Gaussian matrices is not mandated by the problem at hand, and moreover it might be beneficial to choose the signs $\e_1,\ldots,\e_m,\d_1,\ldots,\d_n\in \{-1,1\}$ via a procedure that depends on the given matrix $(a_{ij})$. Nevertheless, Theorem 1.1 shows that oblivious Krivine rounding schemes capture $K_G$ exactly.
 

To date, the best known upper estimates on $K_G$ were obtained via oblivious Krivine schemes. Krivine introduced~\cite{Kri77} this approach to Grothendieck's inequality in 1977 in order to obtain a new upper bound on $K_G$. His argument used a one-dimensional Krivine scheme that is based on a {\em deterministic} partition of $\R$: in Krivine's original work the measure $\mu$ is supported on the pair $(\sign(\cdot),\sign(\cdot))\in \{-1,1\}^{\R}\times \{-1,1\}^{\R}$. Krivine conjectured~\cite{Kri77} that the bound thus obtained is actually the exact value of the Grothendieck constant. Clearly, if true, Krivine's conjecture would have implied that oblivious Krivine schemes yield an exact evaluation of $K_G$, and would have thus made Theorem~\ref{thm:main} redundant.

However, in 2011 Krivine's conjecture was refuted~\cite{BMMN11}, yielding the best known upper bound on $K_G$ via a {\em two-dimensional} Krivine rounding scheme. It was this development that motivated us to investigate the possible validity of Theorem~\ref{thm:main}. While at present there is no compelling conjecture as to the exact value of $K_G$ (though in~\cite{BMMN11} a candidate for this value is proposed based on a certain geometric conjecture that at this point seems quite speculative), Theorem~\ref{thm:main} justifies the focus on oblivious Krivine schemes. In essence, Theorem~\ref{thm:main} says that in order to understand the Grothendieck constant it suffices to focus on ways to partition Euclidean space so as to ensure the validity of~\eqref{eq:scheme identity} with $K$ as small as possible.

\section{Proof of Theorem~\ref{thm:main}}

Fix $k\in \N$. By a well-known duality argument (see~\cite[Thm.~3.4]{Pis12} and~\cite[Prop.~1.1]{Ble12})  there exists a Borel probability measure $\nu_k$ on $\{-1,1\}^{\S^{k-1}}\times \{-1,1\}^{\S^{k-1}}$ such that
\begin{equation}\label{eq:pisier}
\forall\, x,y\in \S^{k-1},\quad \int_{\{-1,1\}^{\S^{k-1}}\times \{-1,1\}^{\S^{k-1}}}f(x)g(y)\nu_k(f,g)=\frac{1}{K_G}\left\langle x,y \right\rangle.
\end{equation}

 By identifying $f\in \{-1,1\}^{\S^{k-1}}$ with its radial extension $x\mapsto f(x/\|x\|_2)$ (with the convention that the radial extension equals $1$ at $x=0$), the measure $\nu_k$ induces a probability measure $\mu_k$ on $\{-1,1\}^{\R^{k}}\times \{-1,1\}^{\R^{k}}$  that satisfies
 \begin{equation}\label{eq:normalized}
 \forall\, x,y\in \R^k\setminus \{0\},\quad \int_{\{-1,1\}^{\R^{k}}\times \{-1,1\}^{\R^{k}}}f(x)g(y)\mu_k(f,g)=\frac{1}{K_G}\left\langle\frac{x}{\|x\|_2},\frac{y}{\|y\|_2} \right \rangle.
 \end{equation}
Our goal will be to show that $\mu_k$ is the desired Krivine scheme. To this end it would be beneficial to rewrite the key requirement~\eqref{eq:scheme identity} using~\eqref{eq:normalized}. Let $G_1,G_2\in \R^k$ be i.i.d.\ standard Gaussian random vectors in $\R^k$. For $t\in [-1,1]$ define
\begin{equation}\label{eq:def f_k expectation}
f_k(t)\eqdef \E_{G_1,G_2}\left[\left\langle\frac{G_1}{\|G_1\|_2},\frac{tG_1+\sqrt{1-t^2}G_2}{\left\|tG_1+\sqrt{1-t^2}G_2\right\|_2} \right \rangle\right].
\end{equation}
Using this notation, combined with~\eqref{eq:normalized} and rotation invariance, the desired identity~\eqref{eq:scheme identity} becomes
\begin{equation}\label{eq:revised scheme identity}
\forall(i,j)\in \{1,\ldots,m\}\times \{1,\ldots,n\},\quad
\frac{1}{K_G}f_k\left(\left\langle x_i',y_j'\right\rangle \right)=\frac{1}{K} \langle x_i,y_j\rangle.
\end{equation}

By direct computation (see~\cite{Haa87} for $k=2$ and~\cite{BOV10} for general $k\ge 2$) for every $t\in [-1,1]$ we have $f_k(t)=
\sum_{n=0}^\infty a_n(k) t^{2k+1}$,
where
\begin{equation}\label{eq:def a_n}
a_n(k)\eqdef \frac{1}{4^n\sqrt{\pi}}\binom{2n}{n}\frac{\Gamma\left((k+1)/2\right)^2\Gamma\left(n+1/2\right)}
{\Gamma\left(k/2\right)\Gamma\left(n+1+k/2\right)}
=\frac{2}{k}\left(\frac{\Gamma\left(\frac{k+1}{2}\right)}{\Gamma\left(\frac{k}{2}\right)}\right)^2
\prod_{j=1}^n\frac{(2j-1)^2}{2j(k+2j)}.
\end{equation}
By Stirling's formula, for every integer $n\ge 2$ we have
\begin{equation}\label{eq:a_n bound}
a_n(k)\asymp \frac{k^{(k+1)/2}}{2^{k/2}\sqrt{n}(n+k/2)^{(k+1)/2}(1+k/(2n))^n}
\lesssim \frac{k^{(k+1)/2}}{k^22^{k/2}(n+k/2)^{(k+1)/2}}.
\end{equation}
Hence,
\begin{equation}\label{eq:2tail}
\sum_{n=2}^\infty a_n(k)\lesssim \frac{k^{(k+1)/2}}{k^22^{k/2}}\int_1^\infty\frac{dx}{(x+k/2)^{(k+1)/2}}\lesssim
\frac{k^{(k+1)/2}}{k^32^{k/2}(k/2)^{(k-1)/2}}\lesssim\frac{1}{k^2}.
\end{equation}
Also,
\begin{equation}\label{eq:first values}
a_0(k)=1-\frac{1}{2k}+O\left(\frac{1}{k^2}\right)\quad\mathrm{and} \quad a_1(k)=\frac{a_0(k)}{2(k+2)}\asymp\frac{1}{k}.
\end{equation}

Write $\D=\{z\in \C: |z|<1\}$. By~\eqref{eq:2tail} for every $z\in \overline \D$ we
have $f_k(z)= \sum_{n=0}^\infty a_n(k)z^{2k+1}$, and $f_k$ is
continuous on $\overline \D$ and analytic on $\D$. Moreover,
by~\eqref{eq:2tail} and~\eqref{eq:first values} we know that there
exists a universal constant $C\in (1,\infty)$ such that
$|f(z)-a_0(k)z|<C/k$ for every $z\in \partial \D$. We assume from now on that $k$ is large enough so that $a_0(k)-2C/k>0$. Then
$|a_0(k)z-f(z)|<|a_0(k)z-\zeta|$ for every $\zeta\in (a_0(k)-C/k)\D$
and $z\in \partial \D$. By Rouch\'e's theorem it follows that
$a_0(k)z-\zeta$ and $f(z)-\zeta$ have the same number of zeros in
$\D$. Hence $f_k^{-1}: (a_0(k)-C/k)\D\to \D$ is well defined and
analytic on $(a_0(k)-C/k)\D$.

For $w\in (a_0(k)-C/k)\D$ write
$f_k^{-1}(w)=\sum_{n=0}^\infty b_n(k)w^{2n+1}$ for some
$\left\{b_n(k)\right\}_{n=0}^\infty\subseteq \R$. Thus $b_0(k)=1/a_0(k)$ and $b_1(k)=-a_1(k)/a_0(k)^4= -1/(2(k+2)a_0(k)^3)$. Observe that for $k$ large enough we have  $1-4C/k\le (a_0(k)-2C/k)/a_0(k)\le \sum_{n=0}^\infty |b_n(k)|(a_0(k)-2C/k)^{2n+1}<\infty$, so that by continuity there exists $c_k\in (0,1-2C/k]$ satisfying $\sum_{n=0}^{\infty} |b_n(k)|c_k^{2n+1}=1-4C/k$.

\begin{lemma}\label{lem:c facts} We have $c_k\ge 1-O(1/k)$.
\end{lemma}

Assuming the validity of Lemma~\ref{lem:c facts} for the moment, we  conclude the proof of Theorem~\ref{thm:main}.

\begin{proof}[Proof of Theorem~\ref{thm:main}] This is a slight variant of Krivine's  original argument~\cite{Kri77}. Define two mappings $S,T:\S^\infty \to \left(\bigoplus_{n=0}^\infty (\ell_2)^{\otimes(2n+1)}\right)\oplus \R^2\eqdef \H\cong\ell_2$ by
$$
S(x)\eqdef \left(\left(|b_{n}(k)|^{1/2}c_k^{(2n+1)/2}x^{\otimes (2n+1)}\right)_{n=0}^\infty,\left(\sqrt{\frac{4C}{k}},0\right)\right),
$$
and
$$
T(x)\eqdef \left(\left(\sign(b_n(k))|b_{n}(k)|^{1/2}c_k^{(2n+1)/2}x^{\otimes (2n+1)}\right)_{n=0}^\infty,\left(0,\sqrt{\frac{4C}{k}}\right)\right).
$$

For every $x\in \S^\infty$ we have $
 \|S(x)\|_\H^2=\|T(x)\|_\H^2=\sum_{n=0}^\infty
|b_{n}(k)|c_k^{2n+1}+4C/k=1$.
Moreover,
$
f_k\left(\langle S(x),T(y)\rangle_\H\right) =f_k\left(\sum_{n=0}^\infty b_{n}(k)c_k^{2n+1}\langle x,y\rangle^{2n+1}\right)= f_k(f_k^{-1}(c_k\langle x,y\rangle))=c_k\langle x,y\rangle.
$
Thus, recalling the formulation of the desired identity appearing in~\eqref{eq:revised scheme identity}, $\mu_k$ is an oblivious $k$-dimensional Krivine scheme of quality $K_G/c_k=(1+O(1/k))K_G$, as desired.
\end{proof}


It remains to prove Lemma~\ref{lem:c facts}.
\begin{proof}[Proof of Lemma~\ref{lem:c facts}]
For every $w\in (a_0(k)-C/k)\D$ write $z=f_k^{-1}(w)\in \D$. Then using~\eqref{eq:first
values},
\begin{multline}\label{eq:three term identity}
f_k^{-1}(w)-b_0(k) w - b_1(k) w^3=f_k^{-1}(w)-\frac{w}{a_0(k)}+\frac{w^3}{2(k+2)a_0(k)^3}\\=
-\frac{f_k(z)-a_0(k)z-a_1(k)z^3}{a_0(k)}-\frac{(a_0(k)z-f_k(z))(a_0(k)^2z^2+a_0(k)zf_k(z)+f_k(z)^2)}
{2(k+2)a_0(k)^3}.
\end{multline}
By~\eqref{eq:2tail} and~\eqref{eq:first values} we have $|f_k(z)-a_0(k)z-a_1(k)z^3| \lesssim 1/k^2$ and $|f_k(z)-a_0(k)z|\lesssim 1/k$. Moreover, due to~\eqref{eq:first values} we have $a_0(k)\asymp 1$ and therefore $|a_0(k)^2z^2+a_0(k)zf_k(z)+f_k(z)^2|\lesssim 1$. Consequently, it follows from~\eqref{eq:three term identity} that
\begin{equation*}
\left|f_k^{-1}(w)-b_0(k) w - b_1(k) w^3\right|\lesssim\frac{1}{k^2}.
\end{equation*}
Hence, by Cauchy's integral formula for every integer $n\ge 2$ and every $r\in (0,a_0(k)-C/k)$,
\begin{equation*}\label{eq:b bound}
|b_n(k)|=\left|\frac{1}{2\pi i}\oint_{r\partial \D}\frac{f_k^{-1}(w)-b_0(k) w - b_1(k) w^3}{w^{2n+2}}dw\right|\lesssim \frac{1}{k^2r^{2n+1}}.
\end{equation*}
Thus there is a universal constant $A\in (0,\infty)$ such that $|b_n(k)|\le Ak^{-2}(a_0(k)-C/k)^{-(2n+1)}$ for   $n\ge 2$. If $c_k\ge a_0(k)-C/k$ then we are done. Assume therefore that $c_k< a_0(k)-C/k$, in which case, recalling that $b_0(k)=1/a_0(k)$ and $b_1(k)= -1/(2(k+2)a_0(k)^3)$, we have
\begin{align}
1=\sum_{n=0}^\infty |b_n(k)| c_k^{2n+1}&\le \frac{c_k}{a_0(k)}+\frac{c_k^3}{2(k+2)a_0(k)^2}+ \frac{A}{k^2}\sum_{n=2}^\infty \left(\frac{c_k}{a_0(k)-C/k}\right)^{2n+1} \nonumber \\
&= \frac{c_k}{a_0(k)}+\frac{c_k^3}{2(k+2)a_0(k)^2}+ \frac{A}{k^2}\cdot \frac{(c_k/(a_0(k)-C/k))^5}{1-(c_k/(a_0(k)-C/k))^2}.\label{eq:delta bound}
\end{align}
By~\eqref{eq:first values}, a straightforward computation shows that~\eqref{eq:delta bound} implies that $c_k\ge 1-O(1/k)$.
\end{proof}

\subsection*{Acknowledgements} We are grateful to Thomas Vidick for helpful suggestions.

\bibliographystyle{abbrv}
\bibliography{krivine}
\end{document}

Let $\H(\R^n)$ denote the space of all closed subsets of $\R^n$,
equipped with the Hausdorff metric.

\begin{theorem}

\end{theorem}

\begin{lemma}
For every $k\in \N$ there exists a probability measure $\mu_k$ on
$\{-1,1\}^{S^{n-1}}\times \{-1,1\}^{S^{n-1}}$ such that
$$
\int_{\{-1,1\}^{S^{k-1}}\times \{-1,1\}^{S^{k-1}}}f(x)g(y)d\mu_k(f,g)=\frac{1}{K_G(n)} \left\langle x,y \right\rangle
$$
\end{lemma}